\documentclass[11pt,a4paper,twoside]{article}
\usepackage{bm, amsmath, amssymb, amsthm} 

\newcommand\I{\mathfrak{T}}

\def\dt{\partial_t}
\def\T{T^\sharp}
\newcommand\mD{{\cal D}}

\def\<{\langle}
\def\>{\rangle}
\def\RR{\mathbb{R}}

\def\eps{\varepsilon}

\newcommand\tr{\operatorname{Tr}}
\newcommand\Div{\operatorname{div}}

\def\Ric{\mathcal Ric}
\def\vol{\operatorname{vol}}

\def\eq{\hspace*{-1.5mm}&=&\hspace*{-1.2mm}}
\def\plus{\hspace*{-1.5mm}&+&\hspace*{-1.2mm}}

\newtheorem{corollary}{Corollary}
\newtheorem{definition}{Definition}
\newtheorem{example}{Example}
\newtheorem{remark}{Remark}
\newtheorem{lemma}{Lemma}
\newtheorem{proposition}{Proposition}
\newtheorem{theorem}{Theorem}

\author{Vladimir Rovenski\footnote{Department of Mathematics, University of Haifa, Mount Carmel, 31905 Haifa,  Israel
       \newline e-mail: {\tt vrovenski@univ.haifa.ac.il} } }

\title{On the mixed scalar curvature of almost multi-product manifolds}

\begin{document}

\date{}

\maketitle

\begin{abstract}
A pseudo-Riemannian manifold endowed with $k>2$ orthogonal complementary distributions
(called a Riemannian almost multi-product structure) appears in such topics as multiply warped products, the webs composed of several foliations,
Dupin hypersurfaces and in stu\-dies of the curvature and Einstein equations.
In this article, we consider the following two problems on the mixed scalar curvature of a Riemannian almost multi-product manifold with a linear connection:
a)~integral formulas and applications to splitting of manifolds,
b)~variation formulas and applications to the mixed Einstein-Hilbert action,
and we gene\-ralize certain results on the mixed scalar curvature of pseudo-Riemannian almost product manifolds.

\vskip.5mm\noindent
\textbf{Keywords}:
Almost multi-product manifold,
mixed scalar curvature,
statistical structure,
mixed Einstein-Hilbert action,
semi-metric connection,
integral formula,
multiply twisted product.

\vskip.5mm
\noindent
\textbf{Mathematics Subject Classifications (2010)} 53C15; 53C12; 53C40

\end{abstract}

\section*{Introduction}

 Distributions on a manifold (that is subbundles of the tangent bundle) appear in various situ\-ations, e.g., \cite{bf,g1967},
and are used to build up notions of integrability, and specifically of a foliated manifold.
The~mixed scalar curvature is the simplest curvature invariant of a Riemannian almost product mani\-fold,
and its study led to many results concerning the existence of foliations and submersions with interesting geometry, for example, integral formulas, splitting results, prescribing the curvature and variational problems, e.g., \cite{arw2014,RWa-1,rze-27b,step1,wa1}.
%
Understanding the properties of this kind of curvature is a deep interesting problem of the geometry of foliations.

In this article, we consider
a smooth connected $n$-dimensional pseudo-Riemannian manifold $(M,g)$ endowed with a linear connection $\bar\nabla$ and
$k>2$ pairwise orthogonal $n_i$-dimensional non-degenerate distributions $\mD_i$ with $\sum n_i=n$.
Such $(M,g,\bar\nabla;\mD_1,\ldots,\mD_k)$, called a \textit{metric-affine almost multi-product manifold},
appears in studies of the curvature and Einstein equations
on {multiply twisted and multiply warped products}, see \cite{chen1,DU,Dimitru,GDH-B,K2020,PK2019,wang,wang2},
in the theory of webs composed of different foliations,~see~\cite{AG2000}, and Dupin hypersurfaces
(which have a constant number of different principal curvatures),
 see~\cite{cecil-ryan,r-IF-k}.
We~define the mixed scalar curvature of $(M,g,\bar\nabla;\mD_1,\ldots,\mD_k)$
and generalize results of the following two problems concerning this kind of curvature of almost product manifolds:

$\bullet$~integral formulas and applications to splitting of manifolds, see~\cite{r-affine,step1,wa1}.

$\bullet$ variation formulas and applications to the mixed Einstein-Hilbert action,

see~\cite{bdrs,bdr,r2018,rz-1,rz-2,rz-3}.

\smallskip\noindent
Integral formulas (usually obtained by applying the Divergence Theorem to appropriate vector fields)
 provide a powerful tool for proving global results in Analy\-sis and Geometry, e.g., \cite{arw2014}.
The~first known integral formula for a closed Riemannian manifold
endowed with a codimension one foliation tells us that the total (i.e., integral) mean curvature
of the leaves vanishes, see~\cite{reeb1}. The~second formula in the series of total $\sigma_k$'s --
elementary symmetric functions of principal curvatures of the leaves -- says that for a codimension one foliation with a unit normal $N$ to the leaves the total $\sigma_2$ is a half of the {total Ricci curvature} in the $N$-direction,~e.g., \cite{arw2014}:
\begin{equation}\label{E-sigma2}
 \int_M (\,\sigma_2-\frac12\,{\rm Ric}_{N,N})\,{\rm d}\vol=0.
\end{equation}
By \eqref{E-sigma2}, there are no totally umbilical foliations on a closed manifold of negative Ricci curvature.
Walczak \cite{wa1} generalized \eqref{E-sigma2} for the total mixed scalar curvature a Riemannian manifold endowed with two complementary orthogonal distributions,
and the integral formula has many applications, e.g., \cite{arw2014,bdrs,bdr,r2018,rz-1,rz-2}.




Many pseudo-Riemannian metrics and linear connections come (as critical points) from variational problems, a~particularly famous of which is the {Einstein-Hilbert action}, e.g., \cite{besse,tr}.
 The~Euler-Lagrange equation for the action when $g$ varies is the {Einstein equation}
with a constant $\Lambda$ (the~``cosmological constant"), the coupling constant $\mathfrak{a}$ and the energy-momentum tensor $\Xi$.
The Euler-Lagrange equation for the action, when a linear connection varies is an algebraic constraint with the torsion tensor
of $\bar\nabla$ and the spin tensor (used to describe the intrinsic angular momentum of particles, e.g.,~\cite{tr}).

On~a manifold equipped with an additional structure related with a distribution (e.g., contact and almost product \cite{bf,g1967}), one can consider an analogue of Einstein-Hilbert action adjusted to that structure,
see \cite{bdrs,bdr,r2018,rz-1,rz-2}.
 We~study
an analog of the Einstein-Hilbert action, where
the scalar curvature
is replaced by the mixed scalar curvature $\overline{\rm S}_{\,\mD_1,\ldots,\mD_k}$, see~\eqref{E-Smix-k}.
This \textit{mixed Einstein-Hilbert action} is a functional of a
metric $g$ and a contorsion tensor $\I:=\bar\nabla-\nabla$ (where $\nabla$ is the Levi-Civita connection):
\begin{equation}\label{Eq-Smix}
 \bar J_\mD : (g,\I) \mapsto\int_{M} \Big\{\frac1{2\mathfrak{a}}\,(\overline{\rm S}_{\,\mD_1,\ldots,\mD_k}-2\,{\Lambda})+{\cal L}\Big\}\,{\rm d}\vol_g\,,
\end{equation}
and its geometrical part is the ``total mixed scalar curvature",
\begin{equation}\label{Eq-Smix-g}
 \bar J^g_\mD : (g,\I)\mapsto\int_{M} \overline{\rm S}_{\,\mD_1,\ldots,\mD_k}\,{\rm d}\vol_g \,.
\end{equation}
 To deal with non-compact manifolds
one may integrate over arbitrarily large, relatively compact domain $\Omega\subset M$, which also contains supports
of variations of $g$ and $\I$.
Varying \eqref{Eq-Smix} with fixed $\I$ in the class of adapted metrics preserving the volume of the manifold, we obtain the Euler-Lagrange equation in the beautiful form of the Einstein equation (for $k=2$, see \cite{bdrs,rz-1,rz-2,rz-3})
\begin{equation}\label{E-gravity}
 \overline\Ric_{\,\mD } -\,(1/2)\,\overline{\cal S}_{\,\mD}\cdot g +\,\Lambda\,g = \mathfrak{a}\cdot\Xi \,,
\end{equation}
where the Ricci tensor and the scalar curvature are replaced by the new
Ricci type tensor
$\overline\Ric_{\,\mD }$, see Section~\ref{sec:EH-action},
and its~trace $\overline{\cal S}_{\mD}$.
 Although generally $\overline\Ric_{\,\mD}$ has a long expression and is not given here, we present it
explicitly for two distinguished classes: statistical connections and semi-symmetric connections.
We also show that a contorsion tensor $\I$ of any adapted statistical connection is critical for \eqref{Eq-Smix} with fixed $g$.
 We suggest that our integral formulas (in Section~\ref{sec:if}) and action \eqref{Eq-Smix} can be useful
 for differential geometry of multiply warped (and twisted) products and in the theory of webs as well as in studying the interaction of several $m$-flows ($m$-dimensional distributions) in the multi-time geometric dynamics, e.g.,~\cite{Neagu,ut}.

 We delegate the following tasks for further study:
a)~find geometrically interesting critical points of the action \eqref{Eq-Smix};
b)~generalize our results for arbitrary variations of metric on an almost multi-product manifold;
c)~find applications of our results in geometry, dynamics and physics.

\section{The mixed scalar curvature}

A~{pseudo-Riemannian metric} $g=\<\cdot,\cdot\>$ of index $q$ on $M$ is an element $g\in{\rm Sym}^2(M)$
of the space of symmetric $(0,2)$-tensors
such that each $g_x\ (x\in M)$ is a {non-degenerate bilinear form of index} $q$ on the tangent space $T_xM$.
For~$q=0$ (i.e., $g_x$ is positive definite) $g$ is a Riemannian metric and for $q=1$ it is called a Lorentz metric.
 A distribution $\mD$ on $(M,g)$ is \textit{non-degenerate},
if $g_x$ is non-degenerate on $\mD_x\subset T_x M$ for all $x\in M$; in this case, the orthogonal complement
of~$\mD ^\bot$ is also non-degenerate, e.g., \cite{bf}.
Denote by ${\rm Riem}(M,\mD_1,\ldots\mD_k)$ the space of pseudo-Riemannian metrics making distributions $\{\mD_i\}$ pairwise orthogonal and non-degenerate.
Let $P_i:TM\to\mD_i$ and $P_{\,i}^\bot:TM\to \mD_{\,i}^\bot$ be orthoprojectors.
The~second fundamental form $h_i:\mD_i\times \mD_i\to \mD_i^\bot$
and the skew-symmetric integrability tensor $T_i:\mD_i\times \mD_i\to \mD_i^\bot$ of $\mD_i$ are defined by
\[
 2\,h_i(X,Y) = P_{\,i}^\bot(\nabla_XY+\nabla_YX),\quad
 2\,T_i(X,Y) = P_{\,i}^\bot(\nabla_XY-\nabla_YX)
             = P_{\,i}^\bot\,[X,Y].
\]
Similarly, $h_{\,i}^\bot,\,H_{\,i}^\bot=\tr_g h_{\,i}^\bot,\,T_{\,i}^\bot$ are
the~second fundamental forms, mean curvature vector fields and the integrability tensors
of distributions $\mD_{\,i}^\bot$.
 Note that $H_i=\sum\nolimits_{\,j\ne i} P_j H_i$, etc.
A~distribution $\mD_i$ is called integrable if $T_i=0$,
and $\mD_i$ is called {totally umbilical}, {harmonic}, or {totally geodesic},
if ${h}_i=({H}_i/n_i)\,g,\ {H}_i =0$, or ${h}_i=0$, respectively, e.g., \cite{bf}.

The \textit{mixed scalar curvature} for a pair of distributions $(\mD,\mD^\bot)$ on a pseudo-Riemannian manifold $(M,g)$ is given~by
\begin{equation*}
 {\rm S}_{\,\mD,\mD^\bot} = \sum\nolimits_{1\le a\le \dim\mD,\,\dim\mD<b\le \dim M}\eps_a \eps_b\,\<R_{\,E_a, E_b} E_a, E_b\>,
\end{equation*}
e.g., \cite{wa1}, where $\{E_1,\ldots,E_n\}$ is a~local $g$-orthonormal frame on $M$ such that
 $E_a\in\mD$ for $1\le a\le \dim\mD$, $E_b\in\mD$ for $\dim\mD<b\le \dim M$ and $\eps_i=\<E_{i},E_{i}\>\in\{-1,1\}$.

The metric-affine~geometry, founded by E.\,Car\-tan, generalizes Riemannian geome\-try:
it considers metric and linear connection as two independent variables
and uses a linear connection $\bar\nabla=\nabla+\I$ instead of the Levi-Civita connection $\nabla$,
 e.g.,~\cite{bf,mikes}.
We~define
auxiliary (1,2)-tensors $\I^*$ and $\I^\wedge$ by
\[
\<\I^*_X Y,Z\> = \<\I_X Z, Y\>,\quad \I^\wedge_X Y = \I_Y X,
\quad X,Y,Z\in\mathfrak{X}_M .
\]
The~following distinguished classes of metric-affine manifolds are considered important.

$\bullet$~\textit{Statistical manifolds}, where the tensor $\bar\nabla g$ is symmetric in all its entries and connection $\bar\nabla$ is  torsion-free
(these conditions are equivalent to $\I^\wedge=\I$ and $\I^* = \I$ or symmetry of the cubic form $A(X,Y,Z):=\<\I_XY,\,Z\>$), constitute an important class of metric-affine manifolds with applications in probability and statistics as well as in information geometry, e.g., \cite{Amari2016,op2016}.

$\bullet$~\textit{Riemann-Cartan manifolds}, where $\bar\nabla g=0$ (i.e., the $\bar\nabla$-parallel transport along the curves preserves the metric), e.g., \cite{gps,rze-27b}.
This is equivalent to $\I^*=-\I$, or the equality $A(X,Y,Z)=A(X,Z,Y)$, and $\bar \nabla$ is then called metric compatible (or, metric connection).
\textit{Semi-symmetric connections} \cite{Yano} constitute a special class of metric connections with applications in geometry and physics.

For the curvature tensor $\bar R_{X,Y}=[\bar\nabla_Y,\bar\nabla_X]+\bar\nabla_{[X,Y]}$ of a linear connection $\bar\nabla$,
we have
\begin{equation*}
\bar R_{X,Y} -R_{X,Y} = (\nabla_Y\,\I)_X -(\nabla_X\,\I)_Y +[\I_Y,\,\I_X],
\end{equation*}
where $R_{X,Y}=[\nabla_Y,\nabla_X]+\nabla_{[X,Y]}$ is the curvature tensor of $\nabla$.
One can also consider the scalar curvature
$\overline{\rm S}= \tr_g\overline{\rm Ric}$, where
$\overline{\rm Ric}(X,Y)=\frac12\tr(Z\to\bar R_{X,Z}\,Y{+}\bar R_{X,Z}\,Y)$
is the Ricci tensor of $\bar\nabla$.

The mixed scalar curvature of $(M,g;\mD_1,\ldots,\mD_k)$ is defined similarly as an averaged mixed sectional curvature.
A plane in $TM$ spanned by two vectors belonging to different distributions $\mD_i$ and $\mD_j$ will be called~\textit{mixed},
and its sectional curvature will be called mixed.
Given $g\in{\rm Riem}(M,\mD_1,\ldots\mD_k)$, there exists a~local $g$-orthonormal frame $\{E_1,\ldots,E_n\}$ on $M$ such that
 $\{E_1,\ldots, E_{n_1}\}\subset\mD_1$ and  $\{E_{n_{\,i-1}+1},\ldots, E_{n_i}\}\subset\mD_i$ for $2\le i\le k$, and $\eps_a=\<E_{a},E_{a}\>\in\{-1,1\}$.
All quantities defined below using such frame do not depend on the choice of this~frame.

\begin{definition}[see \cite{r-IF-k}]\rm
Given $g\in{\rm Riem}(M;\mD_1,\ldots,\mD_k)$
 and a linear connection $\bar\nabla$ on $M$,
the following function on $M$ will be called the \textit{mixed scalar curvature} with respect to $\bar\nabla$:
\begin{equation}\label{E-Smix-k}
 \overline{\rm S}_{\,\mD_1,\ldots,\mD_k}=\sum\nolimits_{\,i<j}\overline{\rm S}\,(\mD_i,\mD_j),
\end{equation}
where
\[
 \overline{\rm S}\,(\mD_i,\mD_j) = \frac12\sum\limits_{\,n_{\,i-1}<a\,\le n_i,\ n_{j-1}<b\le n_j}
 \eps_a\,\eps_b\big(\<\bar R_{E_a,{E}_b}\,E_a,\,{E}_{b}\> + \<\bar R_{E_b,{E}_a}\,E_b,\,{E}_{a}\>\big).
\]
The function ${\rm S}_{\,\mD_1,\ldots,\mD_k}=\sum\nolimits_{\,i<j}{\rm S}(\mD_i,\mD_j)$ on $M$, where
\[
 {\rm S}(\mD_i,\mD_j) = \sum\limits_{\,n_{\,i-1}<a\,\le n_i,\ n_{j-1}<b\le n_j}\eps_a\,\eps_b\<R_{E_a,{E}_b}\,E_a,\,{E}_{b}\> ,
\]
is called the \textit{mixed scalar curvature} (with respect to the Levi-Civita connection $\nabla$).
\end{definition}

\begin{proposition}[see \cite{r-IF-k} for $\I=0$] For any $g\in{\rm Riem}(M,\mD_1,\ldots\mD_k)$ we have
\begin{equation}\label{E-Dk-Smix}
  2\,\overline{\rm S}_{\,\mD_1,\ldots,\mD_k} = \sum\nolimits_{\,i}\overline{\rm S}_{\,\mD_i,\mD ^\bot_i}.
\end{equation}
\end{proposition}

\begin{proof}
For any pair of distributions $(\mD_i,\mD_i^\bot)$ on $(M,g)$ we have
\begin{equation*}
 \overline{\rm S}_{\,\mD_i,\mD_i^\bot} = \sum\nolimits_{\,n_{\,i-1}<a\,\le n_i,\ b\ne(n_{\,i-1}, n_i]}\eps_a \eps_b\,\<\overline R_{\,E_a, E_b} E_a, E_b\>.
\end{equation*}
Thus, \eqref{E-Dk-Smix} directly follows from
$\overline{\rm S}_{\,\mD_i,\mD ^\bot_i}=\sum_{\,j\ne i}\overline{\rm S}_{\,\mD_i,\mD^\bot_j}$
and definition \eqref{E-Smix-k}.
\end{proof}

 The squares of norms of tensors are obtained using
\begin{eqnarray*}
 \<h_i,h_i\>=\sum\nolimits_{\,n_{\,i-1}<a,b\,\le n_i}\eps_a\eps_b\,\<h_i({E}_a,{E}_b),h_i({E}_a,{E}_b)\>, \\
 \<T_i,T_i\>=\sum\nolimits_{\,n_{\,i-1}<a,b\,\le n_i}\eps_a\eps_b\,\<T_i({E}_a,{E}_b),T_i({E}_a,{E}_b)\>.
\end{eqnarray*}
The ``musical" isomorphisms $\sharp$ and $\flat$ will be used for rank one and symmetric rank 2 tensors.
For~example, if $\omega \in\Lambda^1(M)$ is a 1-form and $X,Y\in {\mathfrak X}_M$ then
$\omega(Y)=\<\omega^\sharp,Y\>$ and $X^\flat(Y) =\<X,Y\>$.
For arbitrary (0,2)-tensors $B$ and $C$ we also have $\<B, C\> =\tr_g(B^\sharp C^\sharp)=\<B^\sharp, C^\sharp\>$.
The shape operator $(A_i)_Z$ of $\mD_i$ with $Z\in\mD_i^\bot$ and the operator $(T_i^\sharp)_{Z}$ are defined~by
\[
 \<(A_i)_Z(X),Y\>= \,h_i(X,Y),Z\>,\quad \<(T_i^\sharp)_Z(X),Y\>=\<T_i(X,Y),Z\>, \quad X,Y \in \mD_i .
\]
Similarly, operators $(A_i^\bot)_Z$ and $(T_i^{\bot\sharp})_Z$ with $Z\in\mD_i$ are defined.
The~following formula with the mixed scalar curvature ${\rm S}_{\,\mD,\mD^\bot}$ of $(M,g;\mD,\mD^\bot)$, see \cite{wa1}:
\begin{equation}\label{E-PW0}
 \Div(H+H^\bot) ={\rm S}_{\,\mD,\mD^\bot} +\<h,h\>+\<h^\bot,h^\bot\> -\<H^\bot,H^\bot\>-\<H,H\> -\<T,T\>-\<T^\bot,T^\bot\>
\end{equation}
plays a key role in this work. This can be written shortly as
\begin{equation}\label{E-PW}
 \Div(H+H^\bot) ={\rm S}_{\,\mD,\mD^\bot} -Q(\mD,g)
\end{equation}
using the auxiliary function
\begin{equation}\label{E-func-Q}
 Q(\mD,g) = \<H^\bot,H^\bot\>+\<H,H\> -\<h,h\>-\<h^\bot,h^\bot\> +\<T,T\>+\<T^\bot,T^\bot\> .
\end{equation}
Set $V(\mD)=(\mD\times\mD^\bot)\cup(\mD^\bot\times\mD)$. For a metric-affine manifold we have, see \cite[Lemma 2]{r-affine},
\begin{equation}\label{E-div-barQ}
 \frac12\Div \big(P(\tr_{\,\mD^\bot}(\I -\I^*)) +P^\bot(\tr_{\,\mD}(\I -\I^*)) \big)
 = \bar{\rm S}_{\,\mD,\mD^\bot} - {\rm S}_{\,\mD,\mD^\bot} -\bar Q(\mD,g,\I),
\end{equation}
where
the auxiliary function $\bar Q(\mD,g,\I)$ is given by
\begin{eqnarray}\label{E-barQ}
\nonumber
  2\,\bar Q(\mD,g,\I) = \<\tr_{\,\mD}\I,\,\tr_{\,\mD^\bot}\I^*\> +\<\tr_{\,\mD^\bot}\I,\,\tr_{\,\mD}\I^*\> -\<\I^*,\,\I^\wedge\>_{\,|\,V(\mD)} \\
  +\,\<\tr_{\,\mD}(\I- \I^*) -\tr_{\,\mD^\bot}(\I -\I^*),\, H -H^\bot\>
   +\<\Theta,\ A^\bot -T^{\bot\sharp} + A -T^{\sharp}\> ,
\end{eqnarray}
and
\[
 \Theta = \I -\I^* +\I^\wedge - \I^{* \wedge}.
\]
For instance, if~$\mD$ is spanned by a unit vector field $N$, i.e., $\<N,N\>=\eps_N\in\{-1,1\}$,
then ${\rm S}_{\,\mD,\mD^\bot}=\eps_N{\rm Ric}_{N,N}$, where ${\rm Ric}_{N,N}$ is the Ricci curvature in the $N$-direction.

\begin{remark}\rm
In a local adapted frame, the last term in \eqref{E-barQ} and $\<\I^*,\I^\wedge\>_{\,|\,V(\mD)}$ have the view
\begin{eqnarray*}
   \<\Theta,\ A^\bot -T^{\bot\sharp} + A -T^{\sharp}\>
 = \sum\nolimits_{a\,\le n_1,\,b>n_1}\eps_a\eps_b\,\big(\,\<(\I_{E_b} -\I^*_{E_b}) E_a \\
 +\,(\I_{E_a} -\I^*_{E_a}) {E}_b,\ (A^\bot_{E_a}-T^{\bot\sharp}_{E_a}){E}_b +({A}_{E_b}-T^{\sharp}_{E_b})E_a\>\big) ,\\
  \<\I^*,\ \I^\wedge\>_{\,|\,V(\mD)} = \sum\nolimits_{a\,\le n_1,\,b>n_1}
 \big(\<\I_{E_a} {E}_b,\,\I^*_{E_b} E_a\> + \<\I^*_{E_a} {E}_b,\,\I_{E_b} E_a\>\,\big).
\end{eqnarray*}
\end{remark}

The \textit{divergence} of a vector field $X\in\mathfrak{X}_M$ is given by, e.g., \cite{besse},
\begin{equation}\label{Eq-div}
 (\Div X)\,{\rm d}\vol_g = {\cal L}_{X}({\rm d}\vol_g),
\end{equation}
where ${\rm d} \vol_g$ is the volume form of $g$.
Thus,
$\Div X=\tr(\nabla X)=\Div_{\,i} X+\Div_{\,i}^\bot X$, where
\[
 \Div_{\,i} X =\sum\nolimits_{n_{\,i-1}<a\,\le n_i}\eps_a\,\<\nabla_{E_a}\,X, {E}_a\>,\quad
 \Div_{\,i}^\bot X=\sum\nolimits_{b\ne(n_{\,i-1}, n_i]}\eps_b\,\<\nabla_{E_b}\,X, E_b\> .
\]
Observe that for $X\in\mD_i$ we have
\begin{equation}\label{E-divN}
 {\Div}_i^\bot X = \Div X +\<X,\,H_i^\bot\>.
\end{equation}
More generally, for a $(1,2)$-tensor $P$ we define a $(0,2)$-tensor ${\Div}_i^\bot P$ by
\[
 ({\Div}_i^\bot\,P)(X,Y) = \sum\nolimits_{b\ne(n_{\,i-1}, n_i]}\eps_b\,\<(\nabla_{E_b}\,P)(X,Y), E_b\>,\quad X,Y \in \mathfrak{X}_M.
\]
For a~$\mD_i$-valued $(1,2)$-tensor $P$ we have
${\Div}_i^\bot\,P = \Div P+\<P,\,H_i\>$, where
\begin{equation*}
 ({\Div}_{\,i}\,P)(X,Y) =\sum\nolimits_{n_{\,i-1}<a\,\le n_i}\eps_a\,\<(\nabla_{E_a}\,P)(X,Y), E_a\> = -\<P(X,Y), H_i\>,
\end{equation*}
and $\<P,\,H_i\>(X,Y)=\<P(X,Y),\,H_i\>$ is a $(0,2)$-tensor.
For example, $\Div_{\,i}^\bot h_i = \Div h_i+\<h_i,\,H_i\>$.
 Modifying Divergence theorem on a complete open manifold $(M,g)$ yields the following.

\begin{lemma}[see Proposition~1 in \cite{csc2010}]\label{L-Div-1}
Let $(M^n,g)$ be a complete open Riemannian manifold endowed with a vector field $\xi$
such that $\Div\xi\ge0$. If the norm $\|\xi\|_g\in{\rm L}^1(M,g)$ then $\Div\xi\equiv0$.
\end{lemma}

\begin{definition}[see \cite{r-IF-k}]
The symmetric $(0,2)$-tensor $r(X,Y) = \frac12\sum\nolimits_{\,i} r_{\,\mD_i}(X,Y)$,
where
\begin{equation}\label{E-Rictop2}
 r_{\,\mD_i}(X,Y) = \sum\nolimits_{n_{\,i-1}<a\,\le n_i} \eps_a\, \<R_{\,E_a,\,P_i^\bot\,X}\,E_a, \, P_i^\bot\,Y\>, \quad X,Y\in \mathfrak{X}_M ,
\end{equation}
will be called the \textit{partial Ricci tensor} of $(M,g;\mD_1,\ldots,\mD_k)$ with respect to $\nabla$.
\end{definition}

Similarly to the scalar curvature, the mixed scalar curvature can be seen as the trace:
\[
 {\rm S}_{\,\mD_i,\mD^\bot_i} = \tr_g\,r_{\,\mD_i},\quad
 {\rm S}_{\,\mD_1,\ldots,\mD_k} =\tr_g\,r.
\]
The~partial Ricci tensor in \eqref{E-Rictop2} is presented in the fundamental equality, see \cite{bdr,rz-1},
\begin{equation}\label{E-genRicN}
 r_{\,\mD_i} = \Div h_i +\<h_i,\,H_i\>-{\cal A}_i^\flat-{\cal T}_i^\flat-\Psi_i^\bot+{\rm Def}_{\,\mD_i^\bot}\,H_i^\bot ,
\end{equation}
where the $\mD_i$-\textit{deformation tensor} ${\rm Def}_{\,\mD_i}Z$ of $Z\in\mathfrak{X}_M$
is the symmetric part of $\nabla Z$ restricted to~$\mD_i$,
\begin{equation*}
 2\,{\rm Def}_{\,\mD_i}\,Z(X,Y)=\<\nabla_X Z, Y\> +\<\nabla_Y Z, X\>,\quad X,Y\in \mD_i,
\end{equation*}
and
$(1,1)$-tensors ${\cal A}_i$ (the \textit{Casorati type operator})
and ${\cal T}_i$ and the symmetric $(0,2)$-tensor $\Psi_i$, see \cite{bdr,rz-1}, are defined using operators $A_i$ and $T_i$ by
\begin{eqnarray*}
 && {\cal A}_i=\sum\nolimits_{\,n_{\,i-1}<a\,\le n_i}\eps_a (A_i)_{E_a}^2,\quad
 {\cal T}_i=\sum\nolimits_{\,n_{\,i-1}<a\,\le n_i}\eps_a(T_i^\sharp)_{E_a}^2,\\
 && \Psi_i(X,Y) = \tr((A_i)_Y (A_i)_X+(T_i^\sharp)_Y (T_i^\sharp)_X), \quad X,Y\in\mD_i^\bot.
\end{eqnarray*}
One may derive \eqref{E-PW0} with $\mD=\mD_i$, tracing \eqref{E-genRicN} over $\mD_i$ and applying the equalities
\begin{eqnarray*}
 && \tr_{g}\,({\Div}\,h_i) =\Div H_i,\quad
 \tr\<h_i,\,H_i\> = \<H_i,H_i\>,\quad
 \tr_{g}\Psi_i^\bot =\<h_i^\bot,h_i^\bot\> - \<T_i^\bot,T_i^\bot\>,\\
 && \tr{\cal A}_i = \<h_i,h_i\>,\quad
 \tr{\cal T}_i = -\<T_i,T_i\>,\quad
 \tr_{g}\,({\rm Def}_{\,\mD_i^\bot}\,H_i^\bot) = \Div H_i^\bot +g(H_i^\bot, H_i^\bot) .
\end{eqnarray*}

\section{Integral formulas}
\label{sec:if}

Here, we deduce integral formulas and apply them to obtain splitting results for almost multi-product manifolds with $g>0$.
Define the partial traces of a contorsion tensor $\,\I$ by
\begin{equation*}
 \tr_{\,\mD_i^\bot}\I = \sum\nolimits_{\,b\ne(n_{\,i-1}, n_i]}\eps_b\, \I_{E_b} {E}_b,\quad
 \tr_{\,\mD_i}\I = \sum\nolimits_{\,n_{\,i-1}<a\,\le n_i}\eps_a\, \I_{E_a} \,{E}_{E_a}.
\end{equation*}
The following result generalizes \eqref{E-PW} for arbitrary $k$ (see also \cite{r-IF-k}) and any linear connection $\bar\nabla$.

\begin{proposition}
For a metric-affine almost multi-product manifold $(M,g;\mD_1,\ldots,\mD_k)$ with a linear connection $\bar\nabla=\nabla+\I$ we~have
\begin{eqnarray}\label{E-Q1Q2-gen}
\nonumber
 && \Div\sum\nolimits_{\,i}\Big(\frac12\,\big(P_i(\tr_{\,\mD_i^\bot}(\I -\I^*)) +P_i^\bot(\tr_{\,\mD_i}(\I -\I^*))\big)
 +H_{\,i}+H^\bot_{\,i}\Big) \\
 && = 2\,\bar{\rm S}_{\,\mD_1,\ldots,\mD_k} - \sum\nolimits_{\,i}\big(\bar Q(\mD_i,g,\I) + Q(\mD_i,g)\big),
\end{eqnarray}
where $Q(\mD_i,g)$ is given in \eqref{E-func-Q} with $\mD=\mD_i$
and $\bar Q(\mD_i,g,\I)$ is given in \eqref{E-barQ} with $\mD=\mD_i$.
In~particular, for a Riemannian almost multi-product manifold $(M,g;\mD_1,\ldots,\mD_k)$ we~have
\begin{equation}\label{E-PW3-k}
 \Div \sum\nolimits_{\,i} (H_{\,i}+H^\bot_{\,i}) = 2\,{\rm S}_{\,\mD_1,\ldots,\mD_k} - \sum\nolimits_{\,i}Q(\mD_i,g) .
\end{equation}
\end{proposition}

\begin{proof}
Summing $k$ copies of \eqref{E-PW} with $\mD=\mD_i\ (i=1,\ldots,k)$, and using \eqref{E-Dk-Smix} yields \eqref{E-PW3-k}.
Summing $k$ copies of \eqref{E-div-barQ} with $\mD=\mD_i\ (i=1,\ldots,k)$ and using \eqref{E-Dk-Smix} yields
\begin{eqnarray}\label{E-Q2-gen}
\nonumber
 && \frac12\,\Div\sum\nolimits_{\,i}\big(P_i(\tr_{\,\mD_i^\bot}(\I -\I^*)) +P_i^\bot(\tr_{\,\mD_i}(\I -\I^*)) \big) \\
 && = 2\,\bar{\rm S}_{\,\mD_1,\ldots,\mD_k} - 2\,{\rm S}_{\,\mD_1,\ldots,\mD_k} - \sum\nolimits_{\,i}\bar Q(\mD_i,g,\I).
\end{eqnarray}
Summing \eqref{E-PW3-k} and \eqref{E-Q2-gen} yields \eqref{E-Q1Q2-gen}.
\end{proof}

\begin{theorem}
For a closed manifold $M$ with a metric-affine almost multi-product structure the following integral~formula~holds:
\begin{equation}\label{E-int-k2}
 \int_M\big[ 2\,\overline{\rm S}_{\,\mD_1,\ldots,\mD_k} -\sum\nolimits_{\,i} (Q(\mD_i,g) +\bar Q(\mD_i,g,\I)) \big]\,{\rm d}\vol_g =0\,.
\end{equation}
In~particular, for a closed manifold $M$ with a Riemannian almost multi-product structure we have
\begin{eqnarray}\label{E-int-k}
 \int_M\big(\,2\,{\rm S}_{\,\mD_1,\ldots,\mD_k} - \sum\nolimits_{\,i}Q(\mD_i,g)\,\big)\,{\rm d}\vol_g =0\,.
\end{eqnarray}
\end{theorem}

\begin{proof}
Using Divergence Theorem for \eqref{E-Q1Q2-gen} and \eqref{E-PW3-k} yields \eqref{E-int-k2} and \eqref{E-int-k}, respectively.
\end{proof}

If $\bar\nabla$ is a statistical connection
then $(M,g,\bar\nabla;\mD_1,\ldots,\mD_k)$ is called a \textit{statistical almost multi-product manifold}.

\begin{corollary}\label{C-stat}\rm
For a statistical almost multi-product manifold we have
\[
 2\,\bar Q(\mD_i,g,\I)= 2\,\<\tr_{\,\mD_i}\I,\ \tr_{\,\mD_i^\bot}\I\> -\<\I,\,\I\>_{\,|\,V(\mD_i)}\,;
\]
thus, \eqref{E-Q2-gen} reduces to the equality
\begin{equation}\label{eqvarstat}
 2\,\bar{\rm S}_{\,\mD_1,\ldots,\mD_k} -2\,{\rm S}_{\,\mD_1,\ldots,\mD_k}
 - \!\sum\nolimits_{\,i}\!\big(\<\tr_{\,\mD_i^\bot}\I,\,\tr_{\,\mD_i}\I\> -\frac12\,\<\I,\,\I\>_{\,|V(\mD_i)}\big) = 0,
\end{equation}
and for a closed manifold $M$, \eqref{E-int-k2} reduces to the following integral formula:
\begin{equation*}
 \int_M\!\big[ 2\,\overline{\rm S}_{\,\mD_1,\ldots,\mD_k} -\sum\nolimits_{\,i} \!\big(Q(\mD_i,g) + \<\tr_{\,\mD_i}\I, \tr_{\,\mD_i^\bot}\I\>
 -\frac12\<\I,\,\I\>_{\,|V(\mD_i)} \big) \big]{\rm d}\vol_g \!=0.
\end{equation*}
\end{corollary}

\begin{example}\rm
Let $(M^n,g)$ admits $n$ pairwise orthogonal codimension-one foliations ${\cal F}_i$, and let there exist unit vector fields $N_i$ orthogonal to ${\cal F}_i$.
Let $\sigma_k({\cal F}_i)$ be elementary symmetric functions of principal curvatures of the leaves of ${\cal F}_i$.
Writing down \eqref{E-sigma2} for each $N_i$, summing for $i=1,\ldots,n$ and using ${\rm S}=\sum_{\,i}{\rm Ric}_{N_i,N_i}$, yields the formula
with the scalar curvature {\rm S} of~$(M,g)$,
\begin{equation*}\label{E-sigma2-k}
 \int_M \big(\,2\sum\nolimits_{\,i}\sigma_2({\cal F}_i)-{\rm S}\,\big)\,{\rm d}\vol_g=0,
\end{equation*}
which also follows from a special case of \eqref{E-PW3-k} when $n_i=1$ and $k=n$.
We immediately have two consequences of \eqref{E-sigma2-k}:
a)~if the scalar curvature is nonpositive and not identically zero then each foliation ${\cal F}_i$ cannot be totally umbilical;
b)~if the scalar curvature is nonnegative and not identically zero then each ${\cal F}_i$ cannot be harmonic (i.e., with zero mean curvature of the leaves).
\end{example}

 We say that $(M,g;\mD_1,\ldots,\mD_k)$ \textit{splits} if all distributions  $\mD_i$ are integrable and
$M$ is locally the direct product $M_1\times\ldots\times M_k$ with canonical foliations tangent to $\mD_i$.
It~is well known that if a simply connected manifold splits then it is the direct~product.

We apply the submanifolds theory to
almost multi-product manifolds.

\begin{definition}\rm
A pair $(\mD_i,\mD_j)$ with $i\ne j$ of distributions on
$(M,g;\mD_1,\ldots,\mD_k)$ with $k>2$ is

a) \textit{mixed totally geodesic}, if $h_{\,ij}(X,Y)=0$ for all $X\in\mD_i$ and $Y\in\mD_j$.

b) \textit{mixed integrable}, if $T_{\,ij}(X,Y)=0$ for all $X\in\mD_i$ and $Y\in\mD_j$.
\end{definition}

\begin{lemma}[see \cite{r-IF-k}]\label{L-mixed-YU-TG}
If each pair $(\mD_i,\mD_j)$ with $i\ne j$ on $(M,g;\mD_1,\ldots,\mD_k)$ with $k>2$ is

\smallskip

a$)$ {mixed totally geodesic}, then
$h_{q_1,\ldots,q_r}(X,Y)=0$,

b$)$ {mixed integrable}, then
$T_{q_1,\ldots,q_r}(X,Y)=0$,

\smallskip\noindent
where ${q_1,\ldots,q_r}$ is any subset of $r$ distinct elements of $\{1,\ldots, k\}$
and $X\in\mD_{q(1)},\,Y\in\mD_{q(2)}$.
\end{lemma}


The next definition is introduced to simplify the presentation of results.
 A statistical connection $\bar\nabla=\nabla+\I$ on $(M,g;\mD_1,\ldots,\mD_k)$
is called \textit{adapted} if
 $\I$ is decomposed into $\mD_i$-components, i.e.,
\[
 \I_XY=0\quad (X\in\mD_i,\ \ Y\in\mD_j,\ \ 0\le i\ne j\le k).
\]

\begin{lemma}\label{L-barS=S}
Let $(M,g,\bar\nabla;\mD_1,\ldots,\mD_k)$ be an almost $k$-product manifold with
a statistical adapted connection $\bar\nabla=\nabla+\I$. Then
$\,\overline{\rm S}_{\,\mD_1,\ldots,\mD_k} = {\rm S}_{\,\mD_1,\ldots,\mD_k}$.
\end{lemma}

\begin{proof}
We have $\bar Q(\mD_i,g,\I)=0$ for $1\le i\le k$, see Corollary~\ref{C-stat}.
Thus, \eqref{eqvarstat} yields the claim.
\end{proof}

The following two splitting results generalize \cite[Theorem~6]{step1} and \cite[Theorem~2]{wa1} with $k=2$.

\begin{theorem}
Let an almost $k$-product manifold $(M,g,\bar\nabla;\mD_1,\ldots,\mD_k)$ with $g>0$
and a statistical adapted connection $\bar\nabla=\nabla+\I$ has integrable harmonic distributions
$\mD_1,\ldots,\mD_k$ such that each pair $(\mD_i,\mD_j)$ is mixed integrable.
If~$\,\overline{\rm S}_{\,\mD_1,\ldots,\mD_k}\ge0$, then $(M,g)$ splits.
\end{theorem}

\begin{proof}
From the equality $H_{1\ldots r}=P_{r+1\ldots k}(H_1+\ldots+H_r)$ it follows that
$H_i=0$ for all $i\le k$, then $H_{\,i}^\bot=0$ for all $i\le k$.
Similarly (by Lemma~\ref{L-mixed-YU-TG}), if $T_{\,ij}=0$ for all $i\le k$, then $T_{\,i}^\bot=0$ for all $i\le k$.
By conditions, \eqref{E-func-Q} with $\mD=\mD_i$, \eqref{E-Q1Q2-gen} and Corollary~\ref{C-stat},
\begin{equation*}
 2\,\overline{\rm S}_{\,\mD_1,\ldots,\mD_k} +\sum\nolimits_{\,i} (\|h_i\|^2 + \|h_i^\bot\|^2)
 =0.
\end{equation*}
Thus,
$h_i=0\ (1\le i\le k)$.
By~well-known de Rham decomposition theorem, $(M,g)$~splits.
\end{proof}

\begin{theorem}\label{C-Step3}
Let an almost $k$-product manifold $(M,g,\bar\nabla;\mD_1,\ldots,\mD_k)$ with $g>0$
and a statistical adapted connection $\bar\nabla=\nabla+\I$
has totally umbilical distributions such that each pair $(\mD_i,\mD_j)$ is mixed totally geodesic, $\<H_i,H_j\>=0$ for all $i\ne j$.
If $(M,g)$ is complete open, $\|\xi\|\in{\rm L}^1(M,g)$ for $\xi=\sum_{\,i}(H_i+H_i^\bot)$
and $\,\overline{\rm S}_{\,\mD_1,\ldots,\mD_k}\le0$, then $(M,g)$~splits.
\end{theorem}

\begin{proof}
 By assumptions and Lemma~\ref{L-barS=S}, from
 \eqref{E-Q1Q2-gen} we get
\begin{equation}\label{E-int-k-umb}
 \Div \xi = 2\,\overline{\rm S}_{\,\mD_1,\ldots,\mD_k} -\sum\nolimits_{\,i} Q(\mD_i,g),
\end{equation}
where $Q(\mD_i,g)$ is given in \eqref{E-func-Q} with $\mD=\mD_i$.
By conditions, for any $i\le k$
we have
\begin{equation*}
 \|H_{\,i}^\bot\|^2-\|h_{\,i}^\bot\|^2 = \sum\nolimits_{j\ne i}\frac{n_{j}-1}{n_{j}}\,\|P^\bot_{\,i} H_{j}\|^2 \ge 0,
\end{equation*}
where $P^\bot_{\,i}$ is the orthoprojector onto $\mD_{\,i}^\bot$.
Hence, $Q(\mD_i,g)\ge0$, and from $\overline{\rm S}_{\,\mD_1,\ldots,\mD_k}\le0$ and \eqref{E-int-k-umb} we get $\Div\xi\le0$.
By conditions and Lemma~\ref{L-Div-1}, $\Div\xi=0$. Thus, see \eqref{E-int-k-umb}, $\overline{\rm S}_{\,\mD_1,\ldots,\mD_k}=0$ and ${T}_i$ and $h_i$ vanish.
By de Rham decomposition theorem, $(M,g)$ splits.
\end{proof}

Totally umbilical integrable distributions appear on multiply twisted
products.
A~\textit{multiply twisted product} $F_0\times_{u_1}F_1\times\ldots\times_{u_k} F_k$ of Riemannian manifolds $(F_0,g_{F_0}),\ldots,(F_{k},g_{F_{k}})$
is the direct product $M=F_0\times\ldots\times F_k$ with metric $g=g_{F_0}\oplus u_1^2\,g_{F_1}\oplus\ldots\oplus u_k^2\,g_{F_k}$, where $u_i:F_0\times F_i\to(0,\infty)$ for $1\le i\le k$ are smooth functions, see \cite{wang,wang2}.
The twisted products (i.e., $k=1$, see \cite{pr}) and multiply warped products (i.e., $u_i:F_0\to(0,\infty)$, see \cite{chen1}) are special cases of multiply twisted products.
 Let ${\cal D}_i$ be the distribution on $M$ obtained from the vectors tangent to horizontal lifts of $F_i$.
The {leaves} tangent to ${\cal D}_i,\ i\ge1$, are totally umbilical, with the mean curvature vector fields
 $H_i=-n_iP_0(\nabla(\log u_i))$ tangent to $\mD_0$,
and the {fibers} (i.e., the leaves tangent to ${\cal D}_0$) are totally geodesic, see \cite[Proposition~2.3]{wang2}.

On a multiply twisted product with $k>2$ each pair of distributions is mixed totally geodesic.
Indeed, such manifold is diffeomorphic to the direct product, and the Lie bracket does not depend on metric.

The following corollary of Theorem~\ref{C-Step3} generalizes \cite[Corollary~4]{r-IF-k} with $\I=0$.

\begin{corollary}
Let a closed {multiply twisted product} manifold $(M,g)$ be endowed with a statistical adapted connection
and let $\<H_i,H_j\>=0$ for $i\ne j$.
If $\,\overline{\rm S}_{\,{\cal D}_1,\ldots,{\cal D}_k}\le0$, then $M$ is the direct product.
\end{corollary}


The following result generalizes \cite[Theorem~2]{wa1} and \cite[Corollary~7]{r-affine},
where $k=2$.

\begin{theorem}
Let an almost $k$-product manifold $(M,g,\bar\nabla;\mD_1,\ldots,\mD_k)$ with $g>0$
and a statistical adapted connection $\bar\nabla=\nabla+\I$
has integrable
distributions
$\mD_1,\ldots,\mD_k$ such that each pair $(\mD_i,\mD_j)$ is mixed integrable.
Suppose that $\mD_j$ is harmonic $($i.e., $H_j=0)$ for some index $j$ and $H_i\in\mD_j$ and all $i\ne j$.
If~$\,\overline{\rm S}_{\,\mD_1,\ldots,\mD_k}>0$, then a foliation tangent to $\mD_j$ has no compact leaves.
\end{theorem}

\begin{proof}
By conditions, we have $H_i^\bot\in\mD_j$ for all $i$.
Assume that $\mD_j$ has a compact leaf $L$.
By \eqref{E-divN}, we have $\Div_{L} H_i = \Div H_i +\|H_i\|^2$ for $i\ne j$
and $\Div_{L} H_i^\bot = \Div H_i^\bot +\|H_i^\bot\|^2$ for all $i$.
Thus,
\begin{eqnarray*}
 \Div_{L}\big(\sum\nolimits_{\,i\ne j} H_{\,i}+\sum\nolimits_{\,i} H^\bot_{\,i}\big)
 = \Div\big(\sum\nolimits_{\,i\ne j} H_{\,i}+\sum\nolimits_{\,i}H^\bot_{\,i}\big) \\
 + \sum\nolimits_{\,i\ne j}\|H_{\,i}\|^2+\sum\nolimits_{\,i}\|H^\bot_{\,i}\|^2.
\end{eqnarray*}
Therefore, integrating \eqref{E-Q1Q2-gen} along $L$ and using Lemma~\ref{L-barS=S}, we get
\begin{eqnarray*}
 0 \eq \int_L \Div_{L}\big( \sum\nolimits_{\,i\ne j}H_{\,i}+\sum\nolimits_{\,i} H^\bot_{\,i}\big)\, d\vol_L\\
 \eq \int_L \big(2\,\overline{\rm S}_{\,\mD_1,\ldots,\mD_k} + \sum\nolimits_{\,i} (\|h_{\,i}\|^2+\|h^\bot_{\,i}\|^2) \big)\, d\vol_L > 0.
\end{eqnarray*}
The obtained contradiction completes the proof.
\end{proof}

\section{The Einstein-Hilbert type action}
\label{sec:EH-action}

Here, we consider smooth $1$-parameter variations $g_t,\ |t|<\eps$, of a pseudo-Riemannian metric $g = g_0$
and a linear connection $\bar\nabla^t=\nabla+\I^t$, and derive Euler-Lagrange equations for the Einstein-Hilbert type action.
In Section~\ref{sec:contorsion_zero} we assume that $\I^t\equiv\I^0$ correspond to statistical connections
(in particular, $\bar\nabla^t\equiv\nabla$).
In Section~\ref{sec:contorsion_semi_symmetric} we assume that $\I^t$ correspond to semi-metric connections.

Let infinitesimal variations
 ${B}(t)\equiv\partial g_t/\partial t$,
be supported in a relatively compact domain $\Omega$,
i.e., $g_t=g$ and the symmetric tensor ${B}_t$ va\-nishes outside $\Omega$ for all $t$.
A variation $g_t$ is called \textit{volume-preserving} if ${\rm Vol}(\Omega,g_t) {=} {\rm Vol}(\Omega,g)$ for all~$t$.
 We~adopt the notations $\partial_t \equiv \partial/\partial t,\ {B}\equiv{\dt g_t}_{\,|\,t=0}=\dot g$,
and
write $B$ instead of $B_t$ to make formulas easier to read, where it is not confusing.
We~have
\begin{equation}\label{E-dotvolg}
 \partial_t\,\big({\rm d}\vol_{g}\!\big) = \frac12\,(\tr_{g} B)\,{\rm d}\vol_{g} = \frac12\,\<B,\,g\>\,{\rm d}\vol_{g}.
\end{equation}
Differentiating \eqref{Eq-div} and using \eqref{E-dotvolg}, we obtain
\begin{equation}\label{dtdiv}
 \dt\,(\Div X) = \Div (\dt X) +\frac{1}{2}\,X(\tr_{g} B)
\end{equation}
for any
$t$-dependent vector field $X\in\mathfrak{X}_M$.
By \eqref{E-dotvolg}, \eqref{dtdiv} and the Divergence Theorem, we have
\begin{equation}\label{E-DivThm-2}
 \frac{d}{dt}\int_M (\Div X)\,{\rm d}\vol_g = \int_M \Div\big(\dt X+\frac12\,(\tr_g B) X\big)\,{\rm d}\vol_g = 0
\end{equation}
for any variation $g_t$ with ${\rm supp}\,(\dt g)\subset\Omega$,
and $t$-dependent $X\in\mathfrak{X}_M$ with ${\rm supp}\,(\dt X)\subset\Omega$.

\subsection{Adapted variations of metrics}
\label{sec:contorsion_zero}

A family of metrics
 $\{g_t\in{\rm Riem}(M,\mD_1,\ldots\mD_k):\, |t|<\eps\}$
will be called an \textit{adapted variation}.
In~other words, $\mD_i$ and $\mD_j$ are $g_t$-orthogonal for all $i\ne j$ and all~$t$.
An~adapted variation $g_t$ will be called a $\mD_j$-\textit{variation} (for some $j\le k$) if
the metric changes along $\mD_j$ only, i.e.,
\[
 g_t(X,Y)=g_0(X,Y),\quad X,Y\in\mD_j^\bot,\quad |t|<\eps .
\]
For an adapted variation of metric we have $g_t=g_1(t)\oplus\ldots\oplus g_k(t)$, where $g_j(t) =g_t|_{\,\mD_j}$ are $\mD_j$-variations.
In this case, the tensor $B_t=\dt\, g_t$
on $(M;\mD_1,\ldots,\mD_k)$ is decomposed as
$B_t=\sum_{\,j} {B}_j(t)$, where ${B}_j(t) =\dt g_j(t) =B_t|_{\,\mD_j}$.
Define a self-adjoint $(1,1)$-tensor ${\cal K}_i=\sum\nolimits_{\,n_{\,i-1}<a\,\le n_i}\eps_{\,a}\,[(\T_i)_{E_a},\, (A_i)_{E_a}]$ using the Lie bracket,
For any $(1,2)$-tensors $P_1,P_2$ and a $(0,2)$-tensor $S$
define the $(0,2)$-tensor $\Upsilon_{P_1,P_2}$~by
\[
 \<\Upsilon_{P_1,P_2}, S\> =  \sum\nolimits_{\,\lambda, \mu} \eps_\lambda\, \eps_\mu\,
 \big[S(P_1(e_{\lambda}, e_{\mu}), P_2( e_{\lambda}, e_{\mu})) + S(P_2(e_{\lambda}, e_{\mu}), P_1( e_{\lambda}, e_{\mu}))\big],
\]
where on the left-hand side we have the inner product of $(0,2)$-tensors induced by $g$,
$\{e_{\lambda}\}$ is a local orthonormal basis of $TM$ and $\eps_\lambda = \<e_{\lambda}, e_{\lambda}\>\in\{-1,1\}$.

\begin{remark}\rm
If $g>0$ then $\Upsilon_{\,h_i,h_i}=0$ if and only if $h_i=0$. Indeed, for any $X\in\mD_i$ we~have
\begin{equation*}
 \<\Upsilon_{\,h_i,h_i}, X^\flat\otimes X^\flat\> = 2\sum\nolimits_{\,a,b} \<X, h_i(E_{a}, E_{b})\>^{2},
\end{equation*}
where $\otimes$ denotes the product of tensors.
The above sum is equal to zero if and only if every summand vanishes. This yields $h_i=0$.
Thus, $\Upsilon_{\,h_i,h_i}$ is a ``measure of non-total geodesy" of $\mD_i$.
Similarly,
$\Upsilon_{\,T_i,T_i}$ can be viewed as a ``measure of non-integrability" of $\mD_i$.
\end{remark}

By \eqref{E-Dk-Smix}, in order to derive Euler-Lagrange equations for \eqref{Eq-Smix-g}, we need variations of ${\rm S}_{\,\mD_i,\mD_i^\bot}$,
for which we can use variations of six terms in \eqref{E-func-Q} with $\mD=\mD_i$, given in the following.

\begin{lemma}\label{propvar1}
If $g_t$ is a $\mD_j$-variation of a metric $g\in{\rm Riem}(M,\mD_1,\ldots\mD_k)$, then
\begin{eqnarray}
\nonumber
\label{E-tildeh-gen}
 &&\hskip-9mm
 \dt\<{h}_j^\bot, {h}_j^\bot\> = -\< (1/2)\Upsilon_{h_j^\bot,h_j^\bot},\ B_j\>,\\
\nonumber
 &&\hskip-9mm \dt \<h_j,\ h_j\> = \<\,\Div{h}_j + {\cal K}_j^\flat,\ B_j\> - \Div\<h_j, B_j\>,\\
\nonumber
 &&\hskip-9mm \dt \<{H}_j^\bot, {H}_j^\bot\> = -\<\,{H}_j^\bot\otimes{H}_j^\bot,\ B_j\>,\\
\nonumber
 &&\hskip-9mm \dt \<H_j, H_j\> = \<\,(\Div H_j)\,g_j,\ B_j\> -\Div((\tr_{\,\mD_j} B_j^\sharp) H_j), \\
\nonumber
 &&\hskip-9mm \dt\<{T}_j^\bot, {T}_j^\bot\> = \<\,(1/2)\Upsilon_{T_j^\bot,T_j^\bot},\ B_j\>,\\
 &&\hskip-9mm \dt\<T_j,\ T_j\> = \< 2\,{\cal T}_j^\flat,\ B_j\> ,
\end{eqnarray}
and for $i\ne j$ (when $k>2$) we have dual equations
\begin{eqnarray}
\nonumber
 &&\hskip-9mm
 \dt\<{h}_i, {h}_i\> =  \<-(1/2)\Upsilon_{h_i,h_i},\ B_j\>,\\
\nonumber
\label{E-h-gen}
 &&\hskip-9mm \dt \<h_i^\bot,\ h_i^\bot\> = \<\,\Div{h}_i^\bot + ({\cal K}_i^\bot)^\flat,\ B_j\> - \Div\<h_i^\bot, B_j\>,\\
\nonumber
 &&\hskip-9mm \dt \<{H}_i, {H}_i\> = -\<\,{H}_i\otimes{H}_i,\ B_j\>,\\
\nonumber
 &&\hskip-9mm \dt \<H_i^\bot, H_i^\bot\> = \<\,(\Div H_i^\bot)\,g_j,\ B_j\>
 -\Div((\tr_{\,\mD_i^\bot} B_j^\sharp) H_i^\bot), \\
\nonumber
 &&\hskip-9mm \dt\<{T}_i, {T}_i\> = \<\,(1/2)\Upsilon_{T_i,T_i},\ B_j\>,\\
 &&\hskip-9mm \dt\<T_i^\bot,\ T_i^\bot\> = \< 2\,({\cal T}_i^\bot)^\flat,\ B_j\> .
\end{eqnarray}
\end{lemma}

\begin{proof}
The equations \eqref{E-tildeh-gen} coincide with equations from \cite[Proposition 2]{rz-2} for $(\mD_j,\mD_j^\bot)$,
and equations \eqref{E-h-gen} are dual to \eqref{E-tildeh-gen}.
\end{proof}

\begin{remark}\rm
We cannot vary \eqref{eqvarstat} with respect to metric with fixed $\I$, because when $g$ changes, $\I$ may no longer correspond to statistical connections. 
Instead, we use
variations of \eqref{E-Q1Q2-gen} for $\I$
(see \cite{rz-3} for $k=2$)
corresponding to a statistical connection, for which $\I = \I^* = \I^\wedge$ and~$\Theta=0$.
\end{remark}

By \eqref{E-Dk-Smix}, we need variations of $\overline{\rm S}_{\,\mD_i,\mD_i^\bot}$, for which we can use \eqref{E-div-barQ} with $\mD=\mD_i$.
Variations of terms in \eqref{E-div-barQ} with $\mD=\mD_i$ are given in the following.

\begin{lemma}\label{P-dT-3}
If $g_t$ is a $\mD_j$-variation of a metric $g\in{\rm Riem}(M,\mD_1,\ldots\mD_k)$ on $M$ with a fixed statistical connection $\bar\nabla=\nabla+\I$, then
for $n_{j-1}<a,b\,\le n_j$ and $B_j=\dt g$ we have
\begin{eqnarray}\label{Eq-stat-var}
\nonumber
 &&\hskip-0mm \dt \<\I^*,\ \I^\wedge\>_{\,|\,V(\mD_j)} = -\sum B_j(E_a, E_b) \<\I_{E_a} , \I_{E_b}\>_{\,|\mD_j^\bot} ,\\
\nonumber
 &&\hskip-0mm \dt \<\tr_{\,\mD_j^\bot}\I,\ \tr_{\,\mD_j}\I^*\>
 =0, \\
\nonumber
 &&\hskip-0mm \dt \<\tr_{\,\mD_j}\I^*,\ \tr_{\,\mD_j^\bot}\I\> = -\sum B_j(E_a, E_b)\<\I_{E_a} E_b,\ \tr_{\,\mD_j^\bot}\I \> ,\\
\nonumber
 &&\hskip-0mm \dt \<\Theta, A_j\> = -2 \sum B_j(E_a, E_b) \< (A_j)_{E_a},\, \I_{E_b}\> , \\
\nonumber
 &&\hskip-0mm \dt \<\Theta, T_j^\sharp\> = -2 \sum B_j(E_a, E_b)\< (T^\sharp_j)_{E_a},\, \I_{E_b}\> ,\\
\nonumber
 &&\hskip-0mm \dt \<\Theta, T_j^{\bot\sharp}\>
 =-6\sum B_j(E_a, E_b)\<T_j^\bot(E_b,\,\cdot),\, \I_{E_a}\> ,\\
\nonumber
 &&\hskip-0mm \dt \<\Theta, A_j^\bot \>
 = 4\sum B_j(E_a, E_b)\< h_j^\bot(E_b,\,\cdot),\, \I_{E_a}\>,\\
\nonumber
 &&\hskip-0mm \dt \<\tr_{\,\mD_j}(\I^*-\I),\, H_j^\bot - H_j\> = \sum B_j(E_a, E_b) \big(\<\I_{E_b}\,E_a, H_j^\bot - H_j\> \\
\nonumber
 &&\quad +\<\tr_{\,\mD_j}\I, E_a\>\< H_j, E_b\> \big) , \\
 &&\hskip-0mm \dt\<\,\tr_{\,\mD_j^\bot}(\I^*-\I),\, H_j^\bot - H_j\> =\sum B_j(E_a, E_b) \<\tr_{\,\mD_j^\bot}\I,  E_b\> \<H_j, E_a\> ,
\end{eqnarray}
and for $i\ne j$ (when $k>2$) we have dual equations
\begin{eqnarray}\label{Eq-stat-var-dual}
\nonumber
 &&\hskip-0mm \dt \<\I^*,\ \I^\wedge\>_{\,|\,V(\mD_i^\bot)} = -\sum B_j(E_a, E_b) \<\I_{E_a} , \I_{E_b}\>_{\,|\mD_i} ,\\
\nonumber
 &&\hskip-0mm \dt \<\tr_{\,\mD_i}\I,\ \tr_{\,\mD_i^\bot}\I^*\> =0, \\
\nonumber
 &&\hskip-0mm \dt \<\tr_{\,\mD_i^\bot}\I^*,\ \tr_{\,\mD_i}\I\> = -\sum B_j(E_a, E_b)\<\I_{E_a} E_b,\ \tr_{\,\mD_i}\I \> ,\\
\nonumber
 &&\hskip-0mm \dt \<\Theta, A_i^\bot\> = -2 \sum B_j(E_a, E_b) \< (A_i^\bot)_{E_a},\, \I_{E_b}\> , \\
\nonumber
 &&\hskip-0mm \dt \<\Theta, T_i^{\bot\sharp}\> = -2 \sum B_j(E_a, E_b)\< (T^{\bot\sharp}_i)_{E_a},\, \I_{E_b}\> ,\\
\nonumber
 &&\hskip-0mm \dt \<\Theta, T_i^\sharp\>
 = -6\sum B_j(E_a, E_b)\<T_i(E_b,\,\cdot),\, \I_{E_a}\>,\\
\nonumber
 &&\hskip-0mm \dt \<\Theta, A_i\>
 =4\sum B_j(E_a, E_b)\< h_i(E_b,\,\cdot),\, \I_{E_a}\> , \\
\nonumber
 &&\hskip-0mm \dt \<\tr_{\,\mD_i^\bot}(\I^*-\I),\, H_i - H_i^\bot\> = \sum B_j(E_a, E_b) \big(\<\I_{E_b}\,E_a, H_i - H_i^\bot\> \\
\nonumber
 &&\quad +\<\tr_{\,\mD_i^\bot}\I, E_a\>\< H_i^\bot, E_b\>\big) ,\\
 &&\hskip-0mm \dt\<\,\tr_{\,\mD_i}(\I^*-\I),\, H_i - H_i^\bot\> =\sum B_j(E_a, E_b) \<\tr_{\,\mD_i}\I, E_b\> \<H_i^\bot, E_a\> .
\end{eqnarray}
\end{lemma}

\begin{proof}
Equations \eqref{Eq-stat-var} are a special case (i.e., for adapted variations of metric) of equations from \cite[Lemma 3]{rz-3} for $(\mD_j,\mD_j^\bot)$,
and equations \eqref{Eq-stat-var-dual} are dual to \eqref{Eq-stat-var}.
\end{proof}

From Lemmas~\ref{propvar1} and \ref{P-dT-3} we obtain the following.

\begin{proposition}\label{C-vr-terms}
 For any $\mD_j$-variation of metric $g\in{\rm Riem}(M,\mD_1,\ldots,\mD_k)$
 with a fixed statistical connection $\bar\nabla=\nabla+\I$ we have
\begin{eqnarray}\label{E-dt-barQ}
 \dt\sum\nolimits_{\,i}\bar Q(\mD_i,g_t,\I) \eq\<\bar{\cal Q}(\mD_j,g_t,\I),\,B_j\>\, ,\\
\label{E-dt-Q}
 \dt\sum\nolimits_{\,i} Q(\mD_i,g) \eq \<{\cal Q}(\mD_j,g),\, B_j\> - \Div X_j\,,
\end{eqnarray}
where ${B}_j={\dt g_t}_{\,|\,t=0}$ 
and vector fields $X_j$ and (0,2)-tensors ${\cal Q}(\mD_j,g)$ and $\bar{\cal Q}(\mD_j,g,\I)$ are given~by
\begin{eqnarray*}
 &&\hskip-4mm 2 X_j=\<h_j, B_j\> -(\tr B_j^\sharp) H_j+\sum\nolimits_{\,i\ne j}\,(\<h_i^\bot, B_j\> -(\tr B_j^\sharp) H_i^\bot),\\
 &&\hskip-4mm {\cal Q}(\mD_j,g)=
 -\Div{h}_j -{\cal K}_j^\flat -{H}_j^\bot\otimes{H}_j^\bot +\frac12\Upsilon_{h_j^\bot,h_j^\bot} +\frac12\Upsilon_{T_j^\bot,T_j^\bot}
 +\,2\,{\cal T}_j^\flat + (\Div H_j)\,g_j \\
 && +\sum\nolimits_{\,i\ne j}\big(-\Div{h}_i^\bot {-} ({\cal K}_i^\bot)^\flat {-}{H}_i\otimes{H}_i {+}\frac12\Upsilon_{h_i,h_i} {+}\frac12\Upsilon_{T_i,T_i}
 +\,2\,({\cal T}_i^\bot)^\flat +(\Div H_i^\bot)\,g_j\big), \\
 &&\hskip-4mm \bar{\cal Q}(\mD_j,g,\I)({E}_a,{E}_b) =
 \<\I_{{E}_a}, \I_{{E}_b}\>_{\,|\mD_j^\bot} -\<\I_{{E}_a} {E}_b,\,\tr_{\,\mD_j^\bot}\I\> -2\,\<(A_j)_{{E}_a},\,\I_{{E}_b}\> \\
 && +\, 2\< (T^\sharp_j)_{{E}_a},\, \I_{{E}_b}\> + 6\,\<T_j^\bot({E}_b,\,\cdot),\, \I_{{E}_a}\> + 4\,\< h_j^\bot({E}_b,\,\cdot),\, \I_{{E}_a}\> \\
 && +\,\<\I_{{E}_b} {E}_a, H_j^\bot - H_j\> +\<\tr_{\,\mD_j}\I, {E}_a\>\< H_j, {E}_b\> -\<\tr_{\,\mD_j^\bot}\I, {E}_b\> \<H_j, {E}_a\> \\
 && +\sum\nolimits_{\,i\ne j}\big(\<\I_{{E}_a} , \I_{{E}_b}\>_{\,|\mD_i} -\<\I_{{E}_a} {E}_b,\,\tr_{\,\mD_i}\I\> -2\,\<(A_i^\bot)_{{E}_a},\, \I_{{E}_b}\>  \\
 && +\,2\,\< (T^{\bot\sharp}_i)_{{E}_a},\, \I_{{E}_b}\> + 6\,\< T_i({E}_b,\,\cdot),\, \I_{{E}_a}\> + 4\,\< h_i({E}_b,\,\cdot),\, \I_{{E}_a}\> \\
 && +\,\<\I_{{E}_b} {E}_a, H_i - H_i^\bot\> +\<\tr_{\,\mD_i^\bot}\I, {E}_a\>\< H_i^\bot, {E}_b\> -\<\tr_{\,\mD_i}\I, {E}_b\> \<H_i^\bot, {E}_a\>\big).
\end{eqnarray*}
\end{proposition}


The following theorem (based on Proposition~\ref{C-vr-terms}
and \eqref{E-DivThm-2} for divergence~terms)
with Euler-Lagrange equations for \eqref{Eq-Smix-g}
allows us to find the partial Ricci curvature
(see Corollary~\ref{P-Ric-D}).
Any solution of these Euler-Lagrange equations is called a mixed Einstein structure on an almost multi-product maniofold.

\begin{theorem}\label{T-main01}
A metric $g\in{\rm Riem}(M,\mD_1,\ldots,\mD_k)$ on $M$ with a statistical connection $\bar\nabla=\nabla+\I$ is critical for \eqref{Eq-Smix-g}
with respect to volume-preserving adapted variations if and only if
\begin{equation}\label{ElmixDDvp-b}
  {\cal Q}(\mD_j,g) +\bar{\cal Q}(\mD_j,g,\I) =-\big(\,\overline{\rm S}_{\,\mD_1,\ldots,\mD_k} -\frac12\,\Div\sum\nolimits_{\,i}(H_i + {H}_i^\bot) +\lambda_j\big)\,g_j
\end{equation}
for $1\le j\le k$ and some $\lambda_j\in\RR$.
In the case of $\,\I=0$, \eqref{ElmixDDvp-b} reads as
\begin{equation}\label{ElmixDDvp}
 {\cal Q}(\mD_j,g) =-\big(\,{\rm S}_{\,\mD_1,\ldots,\mD_k} -\frac12\,\Div\sum\nolimits_{\,i}(H_i +  {H}_i^\bot)
 + \lambda_j\big)\,g_j,\quad 1\le j\le k.
\end{equation}
\end{theorem}

\begin{proof} Let a $\mD_j$-variation of $g_t$ (for some $j\le k$) be compactly supported in $\Omega\subset M$.

1. First, we will prove \eqref{ElmixDDvp}.
Let $g_t$ be a $\mD_j$-variation of $g$ compactly supported in $\Omega\subset M$.
Applying Divergence theorem to \eqref{E-dt-Q} and removing integrals of divergences of vector fields supported in $\Omega\subset M$, we get
\begin{equation}\label{E-Sc-var1a}
 \int_{\Omega} \sum\nolimits_{\,i} \dt Q(\mD_i,g_t)_{\,|\,t=0}\,{\rm d}\vol_g = \int_{\Omega} \<{\cal Q}(\mD_j,g),\ B_j\>\,{\rm d}\vol_g .
\end{equation}
By \eqref{E-DivThm-2} with $X=\sum_{\,i}(H_i+H_i^\bot)$ we get
\[
 {\rm\frac{d}{dt}}\int_{\Omega}\Div\sum\nolimits_{\,i}(H_i+H_i^\bot)\,{\rm d}\vol_g =0.
\]
Thus, for $J^g_\mD: g\mapsto\int_{M} {\rm S}_{\,\mD_1,\ldots,\mD_k}\,{\rm d}\vol_g$, i.e., $\bar J^g_\mD$ when $\I=0$, see \eqref{Eq-Smix-g},
using \eqref{E-Sc-var1a} and \eqref{E-dotvolg}, we~get
\begin{eqnarray}
\nonumber
 2\,{\rm\frac{d}{dt}}\,J^g_{\mD}(g_t)_{\,|\,t=0} \eq {\rm\frac{d}{dt}}\int_{\Omega} \sum\nolimits_{\,i} Q(\mD_i,g_t)\,{\rm d}\vol_{g_t \,|\,t=0} \\
\nonumber
 \eq\int_{\Omega} \sum\nolimits_{\,i} \dt Q(\mD_i,g_t)_{|\,t=0}\,{\rm d}\vol_g + \int_{\Omega} \sum\nolimits_{\,i} Q(\mD_i,g)\,\dt({\rm d}\vol_{g_{t}})_{|\,t=0} \\
\label{E-varJh-init2}
 \eq \int_{\Omega}\<\,{\cal Q}(\mD_j,g) +\frac12\sum\nolimits_{\,i} Q(\mD_i,g)\,g,\ B_j\,\>\,{\rm d}\vol_g .
\end{eqnarray}
If $g$ is critical for $J^g_{\mD}$ with respect to $\mD_j$-variations of $g$,
then the integral in \eqref{E-varJh-init2} is zero for any symmetric $(0,2)$-tensor $B_j$.
This yields the $\mD_j$-component of Euler-Lagrange equation
\begin{eqnarray}\label{ElmixDD}
  {\cal Q}(\mD_j,g) +\frac12\sum\nolimits_{\,i} Q(\mD_i,g)\,g_j = 0,\quad 1\le j\le k.
\end{eqnarray}
The Euler-Lagrange equation of \eqref{Eq-Smix-g} with respect to volume-preserving $\mD_j$-variations are given by
${\cal Q}(\mD_j,g) +(\frac12\sum\nolimits_{\,i} Q(\mD_i,g)+\lambda_j)\,g_j = 0$ instead of \eqref{ElmixDD}, i.e., \eqref{ElmixDDvp}.

2. Next, we will prove \eqref{ElmixDDvp-b}.
 For $\bar J^g_\mD$, see \eqref{Eq-Smix-g}, and $\bar Q(\mD_i,g,\I)$ given in \eqref{E-barQ}, we~get
\begin{equation*}
 2\,{\rm\frac{d}{dt}}\,\bar J^g_{\mD}(g_t,\I)_{\,|\,t=0} = {\rm\frac{d}{dt}}\int_{\Omega}
 \sum\nolimits_{\,i}\big(\bar Q(\mD_i,g_t,\I) + Q(\mD_i,g_t)\big)\,{\rm d}\vol_{g_t \,|\,t=0} .
\end{equation*}
Since
\[
 \dt\int_{\Omega} \bar Q(\mD_i,g_t,\I)\,{\rm d}\vol_g
 =\int_{\Omega}\dt \bar Q(\mD_i,g_t,\I)\,{\rm d}\vol_g +\int_{\Omega} \bar Q(\mD_i,g_t,\I)\,\dt({\rm d}\vol_{g_{t}}),
\]
by \eqref{eqvarstat}, \eqref{E-dotvolg} and  \eqref{E-dt-barQ},
 in the case of a statistical connection $\bar\nabla$,
we get
\begin{eqnarray*}
 && 2\,{\rm\frac{d}{dt}}\,\bar J^g_{\mD}(g_t,\I)_{|\,t=0} = \int_{\Omega} \<{\cal Q}(\mD_j,g) +\bar{\cal Q}(\mD_j,g,\I) \\
 && +\,\big(\,\overline{\rm S}_{\,\mD_1,\ldots,\mD_k} - \frac12\,\Div\sum\nolimits_{\,i}(H_i + {H}_i^\bot)\big)\,g_j,\ B_j\big\>\,{\rm d}\vol_g.
\end{eqnarray*}
If $g$ is critical for $\bar J^g_{\mD}$ with respect to $\mD_j$-variations of $g$,
then the above integral is zero for any symmetric $(0,2)$-tensor $B_j$.
This yields the $\mD_j$-component of the Euler-Lagrange equation
\begin{equation}\label{ElmixDD-b}
 {\cal Q}(\mD_j,g) +\bar{\cal Q}(\mD_j,g,\I) + \big(\,\overline{\rm S}_{\,\mD_1,\ldots,\mD_k} - \frac12\,\Div\sum\nolimits_{\,i}(H_i + {H}_i^\bot)\big)\,g_j =0.
\end{equation}
The Euler-Lagrange equation of \eqref{Eq-Smix-g} with respect to volume-preserving $\mD_j$-variations
are \eqref{ElmixDDvp-b} instead of \eqref{ElmixDD-b}.
\end{proof}

\begin{remark}\rm
Theorem~\ref{T-main01} generalizes results in  \cite{rz-2,rz-3} where $k=2$.
Note that \eqref{ElmixDDvp} reads as
\begin{eqnarray}\label{ElmixDDvp-expanded}
 \nonumber
 &&\sum\nolimits_{\,i\ne j}
 \big(\Div{h_i^\bot}+({\cal K}_i^\bot)^\flat +{H}_i\otimes{H}_i-\frac12\Upsilon_{h_i,h_i}-\frac12\Upsilon_{T_i,T_i}-2\,({\cal T}_i^\bot)^\flat\big) \\
\nonumber
 && +\,\Div{h_j} +{\cal K}_j^\flat -\frac12\Upsilon_{h_j^\bot,h_j^\bot}
 +{H}_j^\bot\otimes{H}_j^\bot -\frac12\Upsilon_{T_j^\bot,T_j^\bot} -2\,{\cal T}_j^\flat \\
 && =\big(\,{\rm S}_{\,\mD_1,\ldots,\mD_k} -\Div(H_j + \sum\nolimits_{\,i\ne j} {H}_i^\bot) + \lambda_j\big)\,g_j,\quad 1\le j\le k
\end{eqnarray}
for some $\lambda_j\in\RR$.
One can rewrite \eqref{ElmixDDvp-expanded} using the partial Ricci tensor \eqref{E-Rictop2} and replacing $\Div{h_j}$ and ${\Div}\,h_i^\bot$ for $i\ne j$
in \eqref{ElmixDDvp} due to \eqref{E-genRicN}.
\end{remark}

\begin{example}\rm
A pair $(\mD_i,\mD_j)$ with $i\ne j$ of distributions on a Riemannian almost multi-product manifold
$(M,g;\mD_1,\ldots,\mD_k)$ will be called
\textit{mixed integrable}, if $T_{\,i,j}(X,Y)=0$ for all $X\in\mD_i$ and $Y\in\mD_j$, see \cite{r-IF-k}.
 Let $(M,g;\mD_1,\ldots,\mD_k)$ with $k>2$ has integrable distributions
$\mD_1,\ldots,\mD_k$ and each pair $(\mD_i,\mD_j)$ is mixed integrable.
Then $T_{l}^\bot(X,Y)=0$ for all $l\le k$ and $X\in\mD_{\,i},\,Y\in\mD_{j}$ with $i\ne j$, see \cite[Lemma~2]{r-IF-k}.
In this case, \eqref{ElmixDDvp-expanded} has a shorter view
\begin{eqnarray*}
 &&\sum\nolimits_{\,i\ne j}\big(\Div{h_i^\bot}+({\cal K}_i^\bot)^\flat +{H}_i\otimes{H}_i-\frac12\Upsilon_{h_i,h_i} \big)
 +\Div{h_j} +{\cal K}_j^\flat -\frac12\Upsilon_{h_j^\bot,h_j^\bot} \\
 && +\,{H}_j^\bot\otimes{H}_j^\bot
 =\big({\rm S}_{\,\mD_1,\ldots,\mD_k} -\Div(H_j +\sum\nolimits_{\,i\ne j} {H}_i^\bot) +\lambda_j\big)\,g_j , \quad j=1,\ldots,k.
\end{eqnarray*}
\end{example}

\begin{remark}\rm
According to Lemma~\ref{L-barS=S}, any adapted statistical connection is critical for the action \eqref{Eq-Smix-g}
with respect to variations among adapted statistical connections.
The Euler-Lagrange equations for \eqref{Eq-Smix-g},
considered as a functional of a general contorsion tensor $\I$,
can be obtained using \cite[Theorem~2]{rz-3},
however, unlike Cartan spin-connection equation, these equations also heavily involve the extrinsic Riemannian geometry of the distributions.
\end{remark}

\begin{definition}
\rm
Define the
symmetric $(0,2)$-tensor $\overline\Ric_{\,\mD}$
by its restrictions
on $k$ subbundles $\mD_j$ of $TM$,
\begin{eqnarray}\label{E-main-0ij-kk}
 \overline\Ric_{\,\mD\,|\,\mD_j\times\mD_j} \eq \Ric_{\,\mD\,|\,\mD_j\times\mD_j} - \bar{\cal Q}(\mD_j,g,\I),
 \quad 1\le j\le k ,
\end{eqnarray}
where the symmetric $(0,2)$-tensor $\Ric_{\,\mD}$ is defined by
\begin{equation}\label{E-main-0ij-k}
 \Ric_{\,\mD\,|\,\mD_j\times\mD_j} = -{\cal Q}(\mD_j,g) +\mu_j \,g_j , \quad j=1,\ldots,k ,
\end{equation}
and $\mu_j$ are given in \eqref{E-mu-k} below.
\end{definition}

Using Theorem~\ref{T-main01}, we present the ``mixed Ricci" tensor $\overline\Ric_{\,\mD}$ explicitly for statistical connections.
The~following proposition generalizes results in \cite{r2018,rz-2} when~$k=2$.

\begin{proposition}\label{P-Ric-D}
A metric $g\in{\rm Riem}(M,\mD_1,\ldots\mD_k)$ is critical for the
action \eqref{Eq-Smix-g} with a statistical connection $\bar\nabla=\nabla+\I$ with respect to adapted variations if and only if $g$ and $\I$ satisfy \eqref{E-gravity} and the Ricci type
tensor $\overline\Ric_{\,\mD}$ is given by \eqref{E-main-0ij-kk} and~\eqref{E-main-0ij-k}.
\end{proposition}

\begin{proof} The Euler-Lagrange equations \eqref{ElmixDDvp} consist of $\mD_j\times\mD_j$-components.
Thus, for \eqref{Eq-Smix-g} we obtain \eqref{E-main-0ij-k}.
If $n=2$ (and $k=2$), then we take $\mu_1=\mu_2=0$, see \cite{r2018}.
Assume that $n>2$. Substituting \eqref{E-main-0ij-k} with arbitrary $(\mu_j)$ into \eqref{E-gravity} along $\mD_j$,
we conclude that if the Euler Lagrange equations ${\cal Q}(\mD_j,g)=-b_j\,g_j\ (1\le j\le k)$ hold, where
$b_j\,g_j$ is the RHS of \eqref{ElmixDDvp}, then
 $\Ric_{\,\mD }-(1/2)\,{\cal S}_{\,\mD}\cdot g = 0$,
see \eqref{E-gravity} with $\Lambda=0=\Xi$, if and only if $(\mu_j)$ satisfy the linear~system
\begin{equation}\label{E-mu-system}
  \sum\nolimits_{\,i} n_i\,\mu_i -2\,\mu_j = a_j,\quad j=1,\ldots,k,
\end{equation}
with coefficients $a_j= \tr_g(\sum_{\,i} {\cal Q}(\mD_i,g)) -2\,{\cal Q}(\mD_j,g)$.
The matrix of \eqref{E-mu-system} is
\[
 A=\left(\begin{array}{ccccc}
     n_1-2 & n_2 & \ldots& n_{k-1} & n_k \\
     n_1 & n_2-2 & \ldots& n_{k-1} & n_k \\
     \ldots & \ldots & \ldots& \ldots & \ldots \\
     n_1 & n_2 & \ldots & n_{k-1} & n_k-2 \\
   \end{array}\right) .
\]
Its determinant $\det A=2^{k-1}(2-n)$ is negative when $n>2$. Hence, the system \eqref{E-mu-system} has a unique solution
$(\mu_1,\ldots,\mu_k)$
given by
\begin{equation}\label{E-mu-k}
 \mu_{\,i}=-\frac1{2n-4}\,\big(\sum\nolimits_{\,j}\,(a_{\,i}-a_{j})\,n_{j}-2\,a_{\,i}\big),
\end{equation}
and $\Ric_{\,\mD\,|\,\mD_j\times\mD_j}$ satisfies \eqref{E-main-0ij-k}.
From \eqref{E-main-0ij-k} and \eqref{ElmixDDvp-b} the system \eqref{E-main-0ij-kk} follows.
\end{proof}

\subsection{Semi-symmetric connections}
\label{sec:contorsion_semi_symmetric}

Semi-symmetric connections, that are a special case of a metric connection parameterized by a vector field,
were introduced in~\cite{Yano} and then used by many authors, e.g.,
in \cite{Dimitru,GDH-B,wang,wang2} to examine special multiply Einstein twisted (or warped) products.
 Here, we restrict variations of the mixed scalar curvature  on $(M,g;\mD_1,\ldots\mD_k)$ to semi-symmetric connections
and obtain meaningful Euler-Lagrange equations,
which allow us to present the mixed Ricci tensor $\overline\Ric_{\,\mD}$ explicitly.
Using variations of $\I$ in this class of connections, we also get an analogue of Cartan spin connection equation.

\begin{definition}\rm
A linear connection $\bar\nabla$ on $(M,g)$ is said to be \textit{semi-symmetric} if
\begin{equation}\label{Uconnection}
 \bar\nabla_XY=\nabla_XY + \<U , Y\> X -\<X,Y\>U,\quad X,Y\in\mathfrak{X}_M,
\end{equation}
where $U$
is a given vector field on $M$. In this case, the contorsion tensor is $\I_XY=\<U , Y\> X -\<X,Y\>U$.
\end{definition}

We consider variations of a semi-symmetric connection
only among connections also satisfying \eqref{Uconnection} for some vector fields $U_t$.
The following theorem (concerning a simple version of Einstein-Cartan gravity) generalizes \cite[Theorem~6]{rz-3} for $k=2$.

\begin{theorem}\label{propUconnectionEL}
A pair $(g, \I)$, where $g\in{\rm Riem}(M;\mD_1,\ldots\mD_k)$ and $\I$ corresponds to a semi-symmet\-ric connection on $M$,
is critical for \eqref{Eq-Smix-g} with respect to volume-preserving adapted variations of metric and variations of $\,\I$ corresponding to semi-symmetric connections if and only if the following Euler-Lagrange equations are satisfied:
\begin{eqnarray}\label{UELD}
 \nonumber
 && {\cal Q}(\mD_j,g) -\frac12\,{n_j(n_j^\bot-1)}\,P_j^\bot(U)^\flat\otimes P_j^\bot(U)^\flat +\frac14\,(n_j-n_j^\bot)(\Div P_j^\bot(U))\,g_j \\
\nonumber
 && -\sum\nolimits_{\,i\ne j}\big(\,\frac12\,n_i^\bot(n_i-1)\,P_i(U)^\flat\otimes P_i(U)^\flat -\frac14\,(n_i^\bot-n_i)\Div P_i(U)\,g_j \,\big) \\
 && +\,\big(\,{\rm S}_{\,\mD_1,\ldots,\mD_k} -\frac12\,\Div(H_j +\sum\nolimits_{\,i\ne j} {H}_i^\bot) +\lambda_j\big) g_j = 0,\quad 1\le j\le k
\end{eqnarray}
for some $\lambda_j\in\RR$, and
\begin{equation}\label{UcriticalforI}
 2\,n_j(n_j^\bot-1)\,P_j^\bot(U) = (n_j^\bot-n_j) H_j^\bot ,\quad
 2\,n_j^\bot(n_j-1)\,P_j(U) = (n_j-n_j^\bot) H_j .
\end{equation}
\end{theorem}

\begin{proof}
Let $\bar\nabla$ be a semi-symmetric connection on $(M,g,\mD_i,\mD_i^\bot)$, then \eqref{E-barQ} reduces to
\begin{equation*}
 -2\,\bar Q(\mD_i,g,U) = (n_i^\bot-n_i) \< U , H_i^\bot - H_i \> + n_i^\bot n_i \<U,\, U\> - n_i^\bot \< U^\bot, U^\bot \> - n_i \< U^\top, U^\top \>,
\end{equation*}
see \cite[Lemma~6(a)]{rz-3} for $k=2$.
Let $U_t,\ |t|<\epsilon$, be compactly supported vector fields on $M$, and let $U=U_0$ and $\dot U = \dt U_t |_{t=0}$.
Then, separating parts with $P_i(\dot U)$ and $P_i^\bot(\dot U)$, we get
\begin{eqnarray*}
 2\,\dt\bar Q(\mD_i,g,U_t) |_{\,t=0} =
   \<\,(n_i^\bot-n_i) H_i^\bot -2\,n_i(n_i^\bot-1)\,P_i^\bot(U),\,P_i^\bot(\dot U)\> \\
 +\,\<\,(n_i-n_i^\bot) H_i -2\,n_i^\bot(n_i-1)\,P_i(U),\,P_i(\dot U)\>,
\end{eqnarray*}
see \cite[Proposition~10]{rz-3} for $k=2$, from which \eqref{UcriticalforI} follows.
 For a $\mD_j$-variation of metric, up to divergences of compactly supported vector fields we have, see \cite[Lemma~6(b)]{rz-3} for $k=2$,
\begin{equation*}
 \dt\bar Q(\mD_i,g_t,U) |_{\,t=0} =\left\{
\begin{array}{cc}
 \<B_j,\,\frac14\,(n_i^\bot-n_i)(\Div P_i(U))\,g_j & \\
 \ -\frac12\,n_i^\bot(n_i-1)\,P_i(U)^\flat\otimes P_i(U)^\flat\> , & \quad i\ne j, \\
 \<B_j,\,\frac14\,(n_j-n_j^\bot)(\Div P_j^\bot(U))\,g_j & \\
 \ -\frac12\,n_j(n_j^\bot-1) P_j^\bot(U)^\flat\otimes P_j^\bot(U)^\flat\>, & \quad i = j.
\end{array}\right.
\end{equation*}
From this and \eqref{ElmixDDvp} we get \eqref{UELD}.
\end{proof}

Although $\overline\Ric_{\,\mD}$ in \eqref{E-gravity} has a long expression for arbitrary linear connections $\bar\nabla$ and is not given here explicitly,
for particular case of semi-symmetric connections,
due to Theorem~\ref{propUconnectionEL}, we present the mixed Ricci tensor explicitly
by its restrictions on subbundles $\mD_j$,
\begin{eqnarray*}
 && \overline\Ric_{\,\mD\,|\,\mD_j\times\mD_j} = -{\cal Q}(\mD_j,g) +\mu_j \,g_j
 \\
 && +\,\frac12\,n_j^\bot(n_j-1) P_j(U)^\flat\otimes P_j(U)^\flat -\frac14\,(n_j-n_j^\bot)\big(\Div P_j^\bot(U) + \frac{Z_j}{2-n}\,\big) g_j \\
 && +\sum\nolimits_{\,i\ne j}\big[\,\frac12\,n_i^\bot(n_i-1) P_i^\bot(U)^\flat\otimes P_i^\bot(U)^\flat -\frac14\,(n_i^\bot-n_i)(\Div P_i(U))\,g_j\,\big],
\end{eqnarray*}
also $\overline{\cal S}_\mD = -\tr_g \sum_{\,j}{\cal Q}(\mD_j,g) +\sum_{\,j} n_j\,\mu_j+\frac{2}{2-n}\sum_j Z_j$,
where $\Ric_{\,\mD}$ is given in \eqref{E-main-0ij-k}, $n>2$ and
\begin{eqnarray*}
 Z_j \eq \frac12\,n_j^\bot(n_j-1)\<P_j(U),P_j(U)\>  -\frac14\,n_j(n_j-n_j^\bot)\Div P_j^\bot(U) \\
 \plus\sum\nolimits_{\,i\ne j}\big(\,\frac12\,n_i(n_i^\bot-1)\<P_i^\bot(U),P_i^\bot(U)\> -\frac14\,n_i^\bot(n_i^\bot-n_i)\Div P_i(U)\,\big).
\end{eqnarray*}
This is because $\overline\Ric_{\,\mD}-\frac{1}{2}\,\overline{\rm S}\,_\mD\cdot g=0$ is equivalent to the Euler-Lagrange equations for \eqref{Eq-Smix-g}.

\end{document}